\documentclass[11pt,reqno]{amsart}
\pdfoutput=1
\usepackage[utf8]{inputenc} 
\usepackage[T1]{fontenc}    
\usepackage{hyperref}       
\usepackage{url}            
\usepackage{booktabs}       
\usepackage{amsfonts}       
\usepackage{nicefrac}       
\usepackage{microtype}      
\usepackage[useregional=false]{datetime2} 

\usepackage[english]{babel}
\usepackage{amsmath}
\usepackage{mathrsfs}
\usepackage{amssymb}
\usepackage{graphicx}
\usepackage[colorinlistoftodos]{todonotes}
\usepackage{url}
\usepackage{hyperref}
\usepackage{amsthm}
\usepackage{csquotes}
\usepackage{comment}
\usepackage[ruled]{algorithm}
\usepackage{algorithmicx}
\usepackage[noend]{algpseudocode}
\floatname{algorithm}{Procedure} 
\algnewcommand\algorithmicinput{\textbf{def}}
\algnewcommand\def{\item[\algorithmicinput]}
\algrenewcommand{\algorithmicrequire}{\textbf{Input:}}
\algrenewcommand{\algorithmicensure}{\textbf{Output:}}
\usepackage{enumitem}
\usepackage{wrapfig}
\usepackage{lipsum}
\usepackage{xcolor}
\usepackage{dsfont}
\usepackage{import}
\usepackage{xifthen}
\usepackage{pdfpages}
\usepackage{transparent}

\DeclareMathOperator{\Hyp}{Hyp}
\DeclareMathOperator{\Tree}{Tree}
\newcommand{\EE}{\mathbb{E}}
\newcommand{\PP}{\mathbb{P}}
\newcommand{\KwikCluster}{\textsc{KwikCluster}}

\usepackage[numbers]{natbib}
\usepackage{bibentry}

\theoremstyle{plain}

\newtheorem{thm}{Theorem}[section] 
\newtheorem*{thm-non}{Theorem} 

\theoremstyle{definition}
\newtheorem{lemm}[thm]{Lemma}

\newtheorem{prop}[thm]{Proposition}
\newtheorem{coro}[thm]{Corollary}
\newtheorem{remark}[thm]{Remark}

\begin{document}

\title{A quantitative tree-likeness bound from average hyperbolicity}
\author{Joon-Hyeok Yim}
\address{Yale University, New Haven, CT 06511}
\email{joonhyeokyim@gmail.com}
\thanks{This work was primarily done while the author was at Yale University.}
\date{\today}

\begin{abstract}
    Chatterjee and Sloman proved that a bounded measurable similarity function with sufficiently small average Gromov hyperbolicity admits a tree representation with small mean approximation error. Their argument uses a weighted version of Szemer\'edi's regularity lemma and does not yield explicit quantitative bounds. Here, we establish a tighter relation between average hyperbolicity and mean tree approximation error. For a similarity function $s:S\times S\to[0,b]$, we prove that
    \[ \Tree(s) \leq (63/e)^{1/3} \sqrt[3]{b^2 \Hyp(s)} \leq 2.8512 \sqrt[3]{b^2 \Hyp(s)}.\]
    The proof uses a simple pivoting construction inspired by \KwikCluster. We also discuss the optimal dependence on average hyperbolicity, including a square-root lower bound, and connections with ultrametric fitting.
\end{abstract}

\maketitle

\section{Introduction}
Gromov $\delta$-hyperbolicity is a metric property that captures certain large-scale features of negative curvature in general metric spaces. Since a geodesic metric space can be isometrically embedded into a real tree if and only if it is 0-hyperbolic, $\delta$-hyperbolicity also provides a structural measure of tree-likeness, with smaller values indicating smaller deviations from tree-metric. The relationship between $\delta$-hyperbolicity and tree approximation has been studied in metric geometry~\cite{bridson2013metric, chowdhury2016improved, ghys1990espaces, gromov1987hyperbolic}, network science~\cite{albert2014topological}, and approximation algorithms~\cite{sonthalia2020tree, yim2024fitting}.

Chatterjee and Sloman~\cite{chatterjee2021average} introduced a probabilistic formulation of tree approximation through average Gromov hyperbolicity for similarity functions on probability spaces. Their framework extends the role of Gromov products in metric spaces. On a rooted tree, the Gromov product of two points records the length shared by their paths from the root, providing a similarity representation of the hierarchy. In this setting, average hyperbolicity measures the expected violation of the corresponding three-point inequality, while tree-likeness measures the least mean error in approximating the similarity by such a tree representation. Both quantities are defined using independent samples from the underlying probability measure.

They proved that, given a fixed bound on the similarity function, sufficiently small average hyperbolicity guarantees a tree representation with arbitrarily small mean approximation error. Their result applies also to infinite spaces and nonuniform probability measures, extending the connection between hyperbolicity and tree approximation to a probabilistic setting. The proof uses a weighted version of Szemer\'edi's regularity lemma, and they explicitly raise the question of whether this can be bypassed to obtain useful quantitative bounds. In this paper, we answer this question affirmatively by giving a direct probabilistic construction that yields an explicit polynomial bound. Working with the same notions of average hyperbolicity and tree-likeness, we prove that every bounded similarity function $s: S\times S \to [0,b]$ in this framework satisfies
\[ \Tree(s) \leq (63/e)^{1/3} b^{2/3} \Hyp(s)^{1/3}.\]
In particular, after normalizing $b=1$, the tree approximation error is $O(\Hyp(s)^{1/3})$. Thus the qualitative implication admits an explicit rate in terms of $\Hyp(s)$ and the similarity bound $b$, with no dependence on the cardinality of the underlying space.

An algorithmic approach to ultrametric fitting is to express the input as a similarity function and cluster the resulting threshold relations. Tree metric fitting can also be approached through this framework by using Gromov products with respect to a reference point. At each threshold, this gives a correlation clustering problem, which seeks to minimize disagreements between the threshold relation and a partition~\cite{ailon2008aggregating}. However, partitions obtained separately at different thresholds need not be nested, even when the threshold relations themselves are nested. Representing the similarity by a rooted tree requires the partition at a higher similarity threshold to refine those at lower thresholds. Controlling disagreement while enforcing this consistency is a central issue in hierarchical correlation clustering and its connection with ultrametric fitting under $\ell_1$ objectives~\cite{an2025handling, cohen2024fitting}. As in~\cite{yim2024fitting}, we seek to control the mean approximation error directly in terms of average hyperbolicity. Our formulation on probability spaces also allows the underlying space to be infinite.
\section{Preliminaries}
We use the notions of average hyperbolicity and tree-likeness introduced by Chatterjee and Sloman~\cite{chatterjee2021average}. Let $(S,\mathcal{F},\PP)$ be a probability space with $\mathcal{F}$ countably generated, and let $s:S \times S \to [0,b]$ be a symmetric (i.e., $s(x,y) = s(y,x)$) measurable similarity function. Note that the measure $\PP$ may have atoms.

The average hyperbolicity of $s$ is defined by
\[ \Hyp(s) := \mathbb{E} \left[(\min\{s(X,Z), s(Y,Z)\} - s(X,Y))_{+}\right],
\]
where $(w)_+ = \max(w,0)$ and $X,Y,Z$ are independent random variables with $\PP$.

A rooted (graph-theoretic) tree $T$ with root $r$ is said to be \emph{compatible} with $(S,\mathcal{F})$ if:
\begin{enumerate}[label=(\roman*)]
    \item the set of leaves of $T$ is exactly $S$;
    \item $T\setminus S$ is finite;
    \item for every $v\in T\setminus S$, the set of leaves descending from
    $v$ is measurable.
\end{enumerate}

Unlike \cite{chatterjee2021average}, for convenience, we equip compatible trees with positive edge weights and use the resulting \emph{weighted} shortest path metric $d_T$. For $x,y \in S$, their Gromov product with respect to the root $r$ is
\[ \langle x,y \rangle_r := \frac{d_T(x,r) +d_T(y,r) - d_T(x,y)}{2}.\]
$\langle x,y \rangle_r$ measures the length shared by the paths from $r$ to $x$ and $y$. We require the map $(x,y)\mapsto \langle x,y \rangle_r$ to be $\mathcal{F}\otimes\mathcal{F}$-measurable. In order to do so, we additionally assume that the diagonal $\Delta_S:=\{(x,x): x\in S\} \in \mathcal{F}\otimes\mathcal{F}$\footnote{This assumption is not stated explicitly in~\cite{chatterjee2021average} and does not follow from countable generation alone. The same measurability issue arises in the unweighted formulation as well. For example, take $S=\{a,b\}$ with $\mathcal F=\{\varnothing,S\}$. The tree with root $r$ and leaves $a,b$ satisfies the compatibility conditions, but its Gromov product is not measurable.}. The assumption holds on standard Borel spaces.
We further assume that $a:S \to (0,\infty)$ is $\mathcal{F}$-measurable where $a(x)>0$ denotes the length of the edge incident to leaf $x$. Note that this assumption is automatically satisfied in the unweighted formulation as $a = 1$.

The tree-likeness of $s$ is defined by
\[ \Tree(s) := \inf_{T} \mathbb{E} |s(X,Y) - \langle X,Y \rangle_r|,\]
where $X,Y$ are i.i.d. with $\PP$ and $T$ ranges over all compatible rooted trees together with positive edge lengths satisfying the measurability requirement above.

Note that this weighted formulation has the same infimum as the original definition of Chatterjee and Sloman~\cite{chatterjee2021average}, which uses an unweighted tree $T$ and a global scaling factor $\alpha$. For $\alpha > 0$, assign length $\alpha$ to every edge; the case $\alpha = 0$ follows by taking a limit. Conversely, fix a weighted tree with finite mean approximation error. Since leaf edges contribute only when $X=Y$, integrability allows all but finitely many leaf lengths to be replaced by a common small unit with arbitrarily small mean error. Rounding the remaining edge lengths to positive integer multiples of this unit and subdividing then yields a scaled compatible unweighted tree. Letting the error tend to zero proves equality of the two infima.

\begin{remark}
    This framework also applies to fitting problems for finite metric spaces. For ultrametric fitting, let $b = \max_{x,y} d(x,y)$ and set $s(x,y):= b - d(x,y)$. If $s_T$ is a fitted tree similarity represented by a rooted tree with all leaves at depth $b$, then $d_U := 2(b - s_T)$ is the corresponding leaf-to-leaf ultrametric distance. Since $d_U / 2$ is also an ultrametric, the mean error $\EE |s_T - s| = \EE |d - (d_U/2) |$ measures how well $d$ can be approximated by an ultrametric. Consequently, $\Tree(s)$ measures the optimal mean $\ell_1$ error in approximating $d$ by an ultrametric.
    
    To recover tree fitting problem, fix a reference point $r$ and set $s(x,y):= \frac{d(x,r) + d(y,r) - d(x,y)}{2}$. Then fitting $s$ with the distances to $r$ fixed yields a rooted tree approximation of $d$. This reconstruction is analogous to that used in \textsc{RootedTreeFit}~\cite[Algorithm~5]{yim2024fitting}.
\end{remark}

We are now ready to state our main theorem.

\begin{thm}\label{thm:main}
    Under the above assumptions,
    \[ \Tree(s) \leq (63/e)^{1/3} b^{2/3} \Hyp(s)^{1/3}.\]
\end{thm}

Our result can also be interpreted as a $\ell_1$-distortion bound in the spirit of ultrametric fitting.

\begin{coro}\label{thm:ultrametric-fitting}
    Let $(X,d)$ be a metric space with $|X| = n$, and define
    \[ \operatorname{AvgUM}(d):= \frac{1}{\binom{n}{3}} \sum_{x,y,z \in X} \max_{\pi \; \text{perm}} \left[ d(\pi x, \pi z) - \max(d(\pi x, \pi y), d(\pi y, \pi z)) \right].\]
    Then there exists a randomized polynomial time algorithm that outputs an ultrametric $d_U$ satisfying
    \[ \EE \left[ \sum_{x,y \in \binom{X}{2}} |d(x,y) - d_U(x,y)| \right] \leq n \sqrt[3]{\frac{63}{4e}M^2 \binom{n}{3} \operatorname{AvgUM}(d)} ,\]
    where $M := \max_{x,y \in X}d(x,y) = \operatorname{diam}(X,d)$.
\end{coro}
\section{The Pivot Construction}
In this section, we present our randomized tree construction, detailed in Procedure~\ref{alg:shared-pivot-tree}. The choice of parameters and the error analysis are deferred to Section~4.

Intuitively, at each threshold $t_j$, we consider the binary relation
\[ A_j(x,y):=\mathbf{1}\{s(x,y)\geq t_j\}. \]
A correlation clustering for this relation seeks to place pairs with $A_j(x,y)=1$ in the same cluster and pairs with $A_j(x,y)=0$ in different clusters. To obtain a tree representation, those clusterings must also be consistent across thresholds: the clustering at a higher threshold must refine those at lower thresholds. This is the viewpoint of hierarchical correlation clustering and its connection with ultrametric fitting~\cite{an2025handling, cohen2024fitting, yim2024fitting}. We adopt this viewpoint on a probability space, measuring disagreement by the probability that a randomly sampled pair is classified incorrectly.

Inspired by \textsc{KwikCluster}~\cite{ailon2008aggregating}, our construction forms clusters using sampled pivots and their neighborhoods. Conditional on the sample size $N$, we draw $Z_1, \cdots, Z_N$ independently from the underlying probability measure. Each point $x$ is assigned to its first neighboring pivot, whose index denoted by $\ell_j(x)$. The same ordered sample is used at every threshold. Note that every sampled pivot remains available for assignment, even if that pivot has itself been assigned to an earlier one. Sharing the sample allows us to track assignments across levels.

\begin{algorithm}[t]
    \caption{Tree construction using shared-pivot}
    \label{alg:shared-pivot-tree}
    \begin{algorithmic}[1]
    \Require $(S,\mathcal F,\PP)$ and $s:S\times S\to[0,b]$.
    \Require $0<h\leq b$, $0 \leq \theta<h$, $\lambda>0$, and $\varepsilon>0$.
    \Ensure A compatible rooted weighted tree $T$ with leaf set $S$.
    
    \State $L\gets\lfloor b/h\rfloor$.
    \State Sample $N\sim\operatorname{Poisson}(\lambda)$.
    \State Conditional on $N$, sample $Z_1,\ldots,Z_N$ independently from $\PP$.
    \State Initialize $T$ with root $r$, and set
           $\mathcal P_0\gets\{S\}$, $v_{0,S}\gets r$.
    
    \For{$j=1,\ldots,L$}
        \State $t_j\gets\theta+(j-1)h$.
        \State For each $x\in S$, set
        \Statex \hspace{\algorithmicindent}$\displaystyle
            \ell_j(x)\gets
            \min\{i\in\{1,\ldots,N\}:s(x,Z_i)\geq t_j\},
            \qquad \min\varnothing:=\infty.$
        \State $\displaystyle
            \mathcal P_j\gets
            \{C\cap\ell_j^{-1}(\{i\}):
              C\in\mathcal P_{j-1},\ i\in\{1,\ldots,N\}\}
            \setminus\{\varnothing\}$.
        \ForAll{$C\in\mathcal P_j$}
            \State Let $D\in\mathcal P_{j-1}$ be the unique class containing $C$.
            \State Add a new vertex $v_{j,C}$ to $T$.
            \State Join $v_{j,C}$ to $v_{j-1,D}$ by an edge of length $h$.
        \EndFor
    \EndFor
    
    \ForAll{$x\in S$}
        \State $\displaystyle
            j(x)\gets
            \max\{j\in\{0,\ldots,L\}:x\in\bigcup\mathcal P_j\}$.
        \State Let $C(x)\in\mathcal P_{j(x)}$ be the unique class containing $x$.
        \State Attach $x$ as a leaf to $v_{j(x),C(x)}$
               by an edge of length $\varepsilon$.
    \EndFor
    
    \State \Return $T$.
    \end{algorithmic}
\end{algorithm}

Let $H_j$ denote the partial equivalence relation that groups points with the same finite value of $\ell_j$. Although the threshold relations $A_j$ are nested, the relations $H_j$ need not be. Two points assigned to different pivots at one level may both move to the same later pivot when the threshold increases. We therefore refine the preceding hierarchy by setting
\[
    K_j(x,y):=\prod_{i=1}^{j}H_i(x,y).
\]
It is easy to verify that each class of $K_{j+1}$ is contained in a class of $K_j$.

The resulting nested classes determine the tree. We represent each nonempty class at each level by a distinct internal vertex and connect it to the vertex representing its containing class at the preceding level. We represent $S$ by the root. We assign length $h$ to edges between successive levels and attach each point by a leaf edge of length $\varepsilon$ to the vertex representing its deepest active class, or to the root if it was never active. Since there are only finitely many levels and sampled pivots, the tree has finitely many internal vertices. The classes are measurable by construction. Therefore, the resulting tree is compatible with the framework of Section~2.

The error analysis balances the mesh size against the expected number of sampled pivots. A finer mesh reduces the discretization error but increases the number of levels over which disagreements may accumulate: a pair separated at one level remains separated at subsequent levels of the hierarchy. Increasing the expected number of pivots reduces the contribution of uncovered pairs, while the bounds on pivot-induced disagreements grow with this parameter. We control these disagreements through bad triangles in the threshold relations, whose density integrated over the thresholds equals average hyperbolicity. Section~4 combines the discretization, disagreement, and sampling estimates and chooses the parameters to obtain the cube-root bound.

\begin{figure}
    \centering
    \begin{tikzpicture}[
        scale=0.65,
        x=0.85cm,y=0.85cm,
        font=\scriptsize,
        region/.style={fill=black!8},
        exposed/.style={draw=black,line width=0.75pt},
        covered/.style={draw=black!65,line width=0.65pt,
                        dash pattern=on 2pt off 2.8pt},
        pivot/.style={circle,draw=black,fill=white,line width=0.75pt,
                      inner sep=0pt,minimum size=3.4pt},
        every label/.style={inner sep=2pt}
    ]
    
    \def\SetLevel#1{%
        \ifcase#1\relax
        \or 
            \def\first{(0,0) ellipse[x radius=3.00,y radius=2.60]}
            \def\second{(1.30,0.55) ellipse[x radius=2.65,y radius=1.75]}
            \def\third{(5.30,-0.30) ellipse[x radius=1.10,y radius=1.03]}
            \def\fourth{(4.45,-1.55) ellipse[x radius=0.75,y radius=0.68]}
            \def\fourthlabel{(5.13,-2.20)}
            \def\fifth{(2.65,-0.80) ellipse[x radius=2.00,y radius=1.35]}
        \or 
            \def\first{(0,0) ellipse[x radius=2.20,y radius=1.90]}
            \def\second{(2.65,0.90) ellipse[x radius=0.98,y radius=0.95]}
            \def\third{(5.30,-0.30) ellipse[x radius=0.90,y radius=0.85]}
            \def\fourth{(4.45,-1.55) ellipse[x radius=0.50,y radius=0.46]}
            \def\fourthlabel{(5.00,-2.05)}
            \def\fifth{(2.60,-0.70) ellipse[x radius=1.80,y radius=1.15]}
        \or 
            \def\first{(0,0) ellipse[x radius=1.45,y radius=1.30]}
            \def\second{(2.65,0.90) ellipse[x radius=0.80,y radius=0.75]}
            \def\third{(5.30,-0.30) ellipse[x radius=0.60,y radius=0.57]}
            \def\fourth{(4.45,-1.55) ellipse[x radius=0.32,y radius=0.32]}
            \def\fourthlabel{(4.87,-1.90)}
            \def\fifth{(2.50,-0.60) ellipse[x radius=1.35,y radius=0.90]}
        \fi
    }
    
    \def\FillActiveRegion{%
        \path[region] \first;
        \path[region] \second;
        \path[region] \third;
        \path[region] \fourth;
        \path[region] \fifth;
    }
    
    \def\DrawClassBoundaries{%
        \draw[exposed] \first;
        \begin{scope}[even odd rule]
            \clip (-10,-7) rectangle (10,7) \first;
            \draw[exposed] \second;
        \end{scope}
        \draw[exposed] \third;
        \begin{scope}[even odd rule]
            \clip (-10,-7) rectangle (10,7) \third;
            \draw[exposed] \fourth;
        \end{scope}
        \begin{scope}[even odd rule]
            \clip (-10,-7) rectangle (10,7) \first;
            \clip (-10,-7) rectangle (10,7) \second;
            \clip (-10,-7) rectangle (10,7) \third;
            \clip (-10,-7) rectangle (10,7) \fourth;
            \draw[exposed] \fifth;
        \end{scope}
    }
    
    \def\DrawPivots#1{%
        \node[pivot,label=above left:{$Z_1$}] at (0,0) {};
        \node[pivot,label=above right:{$Z_2$}] at (2.65,0.90) {};
        \node[pivot] at (5.30,-0.30) {};
        \node[inner sep=0pt] at (5.30,0.06) {$Z_3$};
        \node[pivot] at (4.45,-1.55) {};
        \node[anchor=west,inner sep=0pt] at \fourthlabel {$Z_4$};
        \node[pivot,label=below:{$Z_5$}] at (3.45,-0.80) {};
        \node[anchor=east] at (-3.75,0) {$t_{#1}$};
    }
    
    \def\DrawH#1{%
        \SetLevel{#1}
        \FillActiveRegion
        \draw[covered] \second;
        \draw[covered] \fourth;
        \draw[covered] \fifth;
        \DrawClassBoundaries
        \DrawPivots{#1}
    }
    
    \def\DrawK#1{%
        \SetLevel{#1}
        \FillActiveRegion
        \ifnum#1>1\relax
            \begin{scope}
                \clip \first \second \third \fourth \fifth;
                \pgfmathtruncatemacro{\lastlevel}{#1-1}
                \foreach \previouslevel in {1,...,\lastlevel}{%
                    \SetLevel{\previouslevel}
                    \DrawClassBoundaries
                }
            \end{scope}
        \fi
        \SetLevel{#1}
        \DrawClassBoundaries
        \DrawPivots{#1}
    }
    
    \node[font=\large] at (1.20,12.55) {$H_j$};
    \begin{scope}[yshift=8.585cm] \DrawH{3} \end{scope}
    \begin{scope}[yshift=4.505cm] \DrawH{2} \end{scope}
    \begin{scope}                \DrawH{1} \end{scope}
    
    \begin{scope}[xshift=10.88cm]
        \node[font=\large] at (1.20,12.55) {$K_j$};
        \begin{scope}[yshift=8.585cm] \DrawK{3} \end{scope}
        \begin{scope}[yshift=4.505cm] \DrawK{2} \end{scope}
        \begin{scope}                \DrawK{1} \end{scope}
    \end{scope}
    
    \pgfresetboundingbox
    \path[use as bounding box] (-4.20,-2.70) rectangle (19.30,13.00);
    \end{tikzpicture}
    \caption{Pivoting assignments $H_j$ (left) and their hierarchical refinement $K_j$ (right), using the same ordered pivots $Z_i$ at thresholds $t_j$. Points are assigned to their first neighboring pivot, if any.}
    \label{fig:pivoting}
\end{figure}
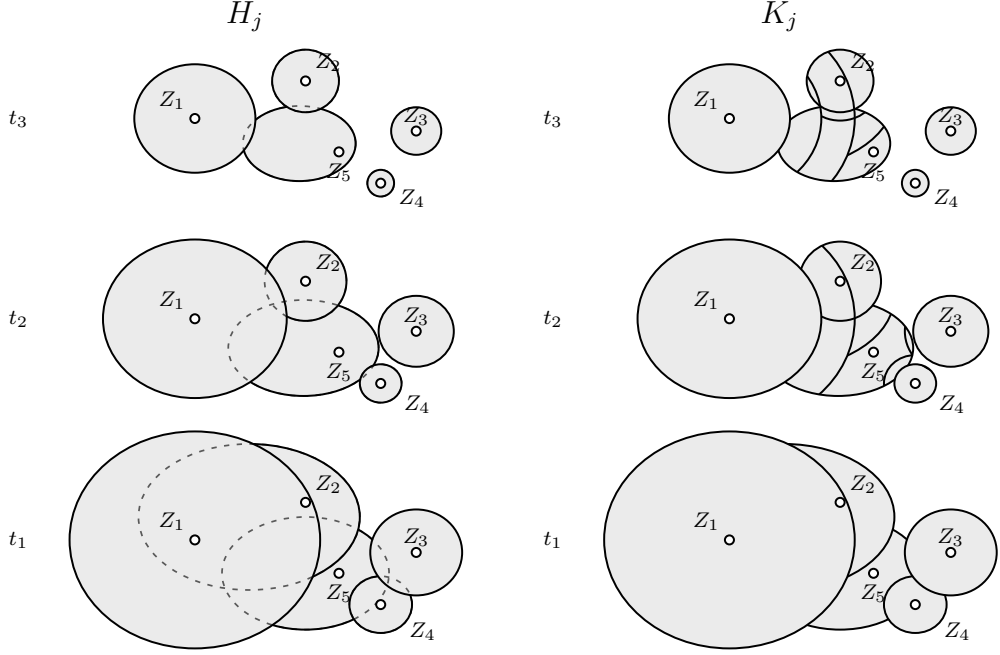

\section{Quantitative Bounds}

We use the notation of Procedure~\ref{alg:shared-pivot-tree}. Let $X,Y,Z$ be independent random points with law $\PP$, independent of the pivot sample. We write $\EE_{\mathrm{piv}}$ for expectation over both $N\sim\operatorname{Poisson}(\lambda)$ and the sampled pivots, and $|\cdot|_1$ for the $L^1(\PP\otimes\PP)$ norm. For $t\in [0,b]$, define
\[ A_t(x,y):=\mathbf{1}\{s(x,y)\geq t\}, \qquad A_j:=A_{t_j}.\]
The clustering at level $j$ and its hierarchical refinement are represented by the indicator functions
\begin{align*}
H_j(x,y)&:=\mathbf{1}\{\ell_j(x)=\ell_j(y)<\infty\},\\
K_j(x,y)&:=\prod_{i=1}^{j}H_i(x,y).
\end{align*}
Thus $K_j(x,y)=1$ precisely when $x$ and $y$ belong to the same class of $\mathcal{P}_1, \cdots, \mathcal{P}_j$, and $K_{j+1}\leq K_j$. We also define the discretized input similarity and the hierarchy similarity by
\[ q(x,y):=h\sum_{j=1}^{L}A_j(x,y), \qquad u(x,y):=h\sum_{j=1}^{L}K_j(x,y). \]
The shifted grid gives $|s(x,y)-q(x,y)|\leq h$ for all $x,y$. For the tree $T$ constructed with leaf-edge length $\varepsilon$, we have
\[ \langle x,y\rangle_r =u(x,y)+\varepsilon\mathbf{1}\{x=y\}. \]
Since $\varepsilon >0$ can be arbitrary, we will take $\varepsilon \searrow 0$ and take the infimum later.

We also define the density of \emph{ordered bad triangles} in $A_t$ by
\[ \beta(t):=\EE\left[ (1-A_t(X,Y))A_t(X,Z)A_t(Y,Z) \right],
\qquad \beta_j:=\beta(t_j). \]
We begin with a basic observation relating the density of bad triangles to average hyperbolicity; cf.~\cite[Proposition~3.2]{yim2024fitting}.
\begin{lemm}\label{lemma:beta-bad-triangle}
    We have \[ \int_0^b \beta(t)dt = \Hyp(s).\]
\end{lemm}
\begin{proof}
    For every $x,y,z \in S$,
    \[ (1-A_t(x,y))A_t(x,z)A_t(y,z) =\mathbf{1}\{s(x,y)<t\leq\min\{s(x,z),s(y,z)\}\}. \]
    Consequently,
    \begin{align*}
        \int_0^b \beta(t)dt
        &= \EE\left[
            \int_0^b
            \mathbf{1}\{s(X,Y)<t\leq\min\{s(X,Z),s(Y,Z)\}\}
            dt
        \right]\\
        &= \EE\left[
            \bigl(\min\{s(X,Z),s(Y,Z)\}-s(X,Y)\bigr)_+
        \right]\\
        &= \Hyp(s),
    \end{align*}
    as desired.
\end{proof}
For the error analysis, we bound the tree approximation error given a fixed realization of the pivots as follows:
\begin{align}\label{eq:s-u}
    \EE_{X,Y} |s(X,Y) - \langle X,Y \rangle_r | & \leq \EE_{X,Y} |s(X,Y) - u(X,Y)| + \varepsilon \nonumber \\
    & \leq \EE_{X,Y} |s(X,Y) - q(X,Y)| + \EE_{X,Y} |q(X,Y) - u(X,Y)| + \varepsilon \nonumber \\
    &\leq h+h\sum_{j=1}^{L}
    \PP\bigl(A_j(X,Y)\neq K_j(X,Y)\bigr)+\varepsilon \tag{*}.
\end{align}
We therefore focus on bounding $\PP\bigl(A_j(X,Y)\neq K_j(X,Y)\bigr)$. Fix $0<h\leq b$, $\theta\in [0,h)$, and $\lambda>0$.
\begin{prop}\label{prop:disagreement}
For the shared-pivot construction with
$N\sim\operatorname{Poisson}(\lambda)$, we have
\[
    \EE_{\mathrm{piv}}\left[
        \PP\bigl(A_j(X,Y)\neq K_j(X,Y)\bigr)
    \right]
    \leq \lambda\beta_j
    +2\lambda\sum_{i=1}^{j}\beta_i
    +\frac{1}{e\lambda}.
\]
for every $j = 1, \cdots, L$. Here the inner probability is taken over $X,Y$ with the pivots fixed.
\end{prop}
\begin{proof}
    We first condition on $N = k$ and bound $\EE_{\mathrm{piv}}\left[\PP\bigl(A_j(X,Y)\neq K_j(X,Y)\bigr) \,\middle|\,N=k \, \right]$. We will check the three cases when $A_j$ and $K_j$ differs.

    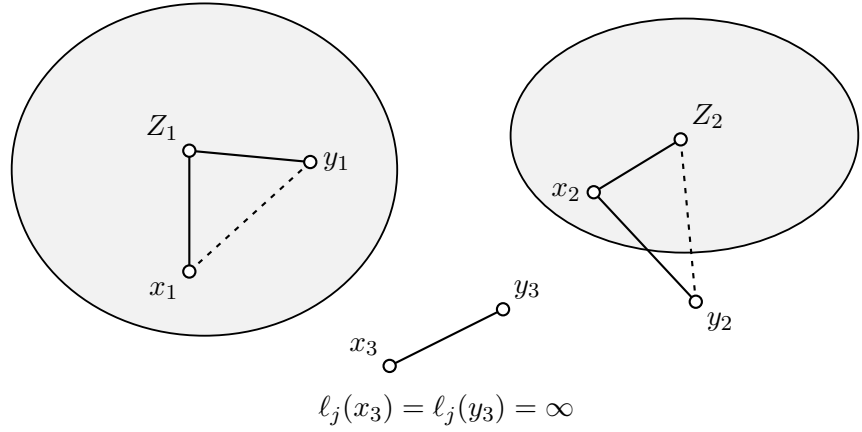
\begin{figure}
    \centering
    \begin{tikzpicture}[
        x=1cm, y=1cm,
        font=\normalsize,
        cluster/.style={draw=black, fill=black!5, line width=0.7pt},
        positive/.style={draw=black, line width=0.8pt},
        negative/.style={draw=black, line width=0.8pt,
                         dash pattern=on 2pt off 2.6pt},
        vertex/.style={circle, draw=black, fill=white,
                       line width=0.8pt, inner sep=0pt, minimum size=4.6pt},
        every label/.style={inner sep=2pt}
    ]
    
    \draw[cluster] (2.70,2.90) ellipse[x radius=2.55,y radius=2.20];
    \draw[cluster] (9.05,3.35) ellipse[x radius=2.30,y radius=1.55];
    
    \coordinate (z1) at (2.50,3.15);
    \coordinate (x1) at (2.50,1.55);
    \coordinate (y1) at (4.10,3.00);
    \draw[positive] (x1) -- (z1) -- (y1);
    \draw[negative] (x1) -- (y1);
    \node[vertex,label=above left:{$Z_1$}] at (z1) {};
    \node[vertex,label=below left:{$x_1$}] at (x1) {};
    \node[vertex,label=right:{$y_1$}] at (y1) {};
    
    \coordinate (z2) at (9.00,3.30);
    \coordinate (x2) at (7.85,2.60);
    \coordinate (y2) at (9.20,1.15);
    \draw[positive] (z2) -- (x2) -- (y2);
    \draw[negative] (z2) -- (y2);
    \node[vertex,label=above right:{$Z_2$}] at (z2) {};
    \node[vertex,label=left:{$x_2$}] at (x2) {};
    \node[vertex,label=below right:{$y_2$}] at (y2) {};
    
    \coordinate (x3) at (5.15,0.30);
    \coordinate (y3) at (6.65,1.05);
    \draw[positive] (x3) -- (y3);
    \node[vertex,label=above left:{$x_3$}] at (x3) {};
    \node[vertex,label=above right:{$y_3$}] at (y3) {};
    \node at (5.90,-0.28) {$\ell_j(x_3)=\ell_j(y_3)=\infty$};
    
    \end{tikzpicture}

    \caption{Three cases of disagreement: the pairs $(x_1,y_1)$, $(x_2,y_2)$, and $(x_3,y_3)$ illustrate a false positive, a false negative with a separating pivot, and an unassigned false negative, respectively.}
    \label{fig:disagreement}
\end{figure}
    
    \begin{description}
        \item[Case 1] $A_j(X,Y) = 0$ with $K_j(X,Y) = 1$: $H_j(X,Y) = 1$ as well. Therefore, there exists $Z_i \,(i=1,\cdots,k)$ such that $A_j(X,Z_i) =  A_j(Y,Z_i) = 1$. Therefore, a union bound gives
        \begin{align*}
            &\EE_{\mathrm{piv}}\left[
                \PP\bigl(A_j(X,Y)=0,  K_j(X,Y)=1\bigr)
                \,\middle|\,N=k
            \right]
            \\\leq&
            \sum_{i=1}^{k}\EE_{X,Y,Z_i}\left[
                (1-A_j(X,Y))A_j(X,Z_i)A_j(Y,Z_i)
                \,\middle|\,N=k
            \right]=k\beta_j.
        \end{align*}
        Intuitively, the probability of a false positive can be bounded in terms of the density of bad triangles. Especially, such a witness consists of the two points and their common pivot, which is one of the $k$ sampled pivots.
        \item[Case 2] $A_j(X,Y) = 1$ with $K_j(X,Y) = 0$ and $\min(\ell_j(X), \ell_j(Y)) < \infty$: By definition, $\ell_m \leq\ell_{m+1}$ so that $\min\{\ell_{j'}(X),\ell_{j'}(Y)\}<\infty$ for every $j'\leq j$. Since $K_j(X,Y)=0$, there exists $j'\leq j$ such that $H_{j'}(X,Y)=0$. WLOG assume $\ell_{j'}(X) < \ell_{j'}(Y)$. Then the pivot $Z_i$, where $i=\ell_{j'}(X)$, separates $X$ and $Y$. Since $A_j\leq A_{j'}$, we have $A_{j'}(X,Y) = A_{j'}(X,Z_i) = 1, \, A_{j'}(Y,Z_i) = 0$.
        Again, a union bound gives
        \begin{align*}
            &\EE_{\mathrm{piv}}\left[
                \PP\bigl(
                    A_j(X,Y)=1,\ K_j(X,Y)=0,\
                    \min\{\ell_j(X),\ell_j(Y)\}<\infty
                \bigr)
                \,\middle|\,N=k
            \right]\\
            \leq&
            2\cdot \sum_{j'=1}^{j}\sum_{i=1}^{k}
            \EE_{X,Y,Z_i}\left[
                A_{j'}(X,Y)A_{j'}(X,Z_i)(1-A_{j'}(Y,Z_i))
                \,\middle|\,N=k
            \right]\\
            =&2k \cdot \sum_{j'=1}^{j}\beta_{j'}.
        \end{align*}
        Intuitively, the probability of a false negative caused by a \emph{separating pivot} can also be bounded in terms of the density of bad triangles containing sampled pivots. Since we limit the number of sampled points, both points may remain unassigned, giving rise to another type of false negative. We will treat this in the next case.
        \item[Case 3] $A_j(X,Y) = 1$ with $\ell_j(X) = \ell_j(Y) = \infty$: Since $\ell_j(X) = \infty$, $A_j(X, Z_i) = 0$ for all $i = 1,\cdots,k$. Define the neighborhood mass
        \[ q_j(x) := \int_S A_j(x,z) d \PP(z).\]
        Then conditional on $X$ and $N=k$, the events $A_j(X,Y) = 1$ and $A_j(X, Z_i) = 0$ for all $i=1,\cdots,k$ are independent, with probabilities $q_j(X)$ and $(1 - q_j(X))^k$. respectively. Dropping the condition $\ell_j(Y) = \infty$ and averaging over $X$, we obtain
        \[\EE_{\mathrm{piv}}\left[ \PP\bigl( A_j(X,Y)=1,\ \ell_j(X)=\ell_j(Y)=\infty \bigr) \,\middle|\,N=k \right] \leq \EE_X\left[q_j(X)(1-q_j(X))^k\right].\]
    \end{description}
    Finally, we take the expectation over $N$. It is clear to see that the probability of Case 1 and Case 2 being bounded by $\lambda \beta_j$ and $2 \lambda \sum_{i=1}^{j} \beta_i$, respectively. To verify Case 3, we see that
    \begin{align*}
        & \EE_{\mathrm{piv}}\left[ \PP\bigl( A_j(X,Y)=1,\ \ell_j(X)=\ell_j(Y)=\infty \bigr) \right] \\
        = & \sum_{k=0}^{\infty} \PP(N=k) \cdot \EE_{\mathrm{piv}}\left[ \PP\bigl( A_j(X,Y)=1,\ \ell_j(X)=\ell_j(Y)=\infty \bigr) \,\middle|\,N=k \right] \\
        \leq & \sum_{k=0}^{\infty} \frac{\lambda^k e^{-\lambda}}{k!} \cdot \EE_X \left[ q_j(X)(1-q_j(X))^k \right] \\
        = & \EE_X \left[ \sum_{k=0}^{\infty} \frac{\lambda^k e^{-\lambda}}{k!} q_j(X)(1-q_j(X))^k \right] = \EE_X q_j(X) e^{-\lambda q_j(X)} \leq \sup_{q \geq 0} qe^{-\lambda q} = \frac{1}{e \lambda}.
    \end{align*}
    Combining the three cases yields the desired bound on $\EE_{\mathrm{piv}}\left[ \PP\bigl(A_j(X,Y)\neq K_j(X,Y)\bigr) \right]$, completing the proof.
\end{proof}
\begin{thm}\label{thm:treerealization}
    Fix $0<h \leq b$, $\lambda, \varepsilon >0$ and $\theta \in [0,h)$, and set $L := \lfloor b/h \rfloor$. Procedure~\ref{alg:shared-pivot-tree} outputs a random compatible rooted tree $T$ satisfying
    \[ \EE_{\mathrm{piv}}\left[ \EE_{X,Y}|s(X,Y)-\langle X,Y\rangle_r|\right] \leq h + \lambda h(2L+1) \left(  \sum_{j=1}^{L} \beta_j \right) + \frac{hL}{e \lambda} + \varepsilon.\]
\end{thm}
\begin{proof}
    We combine the error decomposition \ref{eq:s-u} with Proposition~\ref{prop:disagreement}. We obtain
    \[
        \EE_{\mathrm{piv}}\left[ \EE_{X,Y}|s(X,Y)-\langle X,Y\rangle_r|\right] \leq
        h+\lambda h\sum_{j=1}^{L}
        \left(\beta_j+2\sum_{i=1}^{j}\beta_i\right)
        +\frac{hL}{e\lambda}+\varepsilon \leq
        h+\lambda h(2L+1)\left(\sum_{j=1}^{L}\beta_j \right)
        +\frac{hL}{e\lambda}+\varepsilon.
    \]
\end{proof}
We now express the approximation error in terms of average hyperbolicity and the parameters of our construction. The following bound makes explicit the roles of the mesh threshold $h$ and the expected number of pivots $\lambda$. A suitable choice of these two parameters yields our main theorem.
\begin{thm}\label{thm:treebound}
    Given $b, 0 < h \leq b$ and $\lambda > 0$, we have
    \[ \Tree(s) \leq h + \lambda \Hyp(s) \left(\frac{2b}{h} + 1\right) + \frac{b}{e\lambda}.\]
\end{thm}
\begin{proof}
    Since $hL \leq b$, we have
    \[ \int_0^h \left( \sum_{j=1}^{L}\beta(\theta + (j-1)h) \right) d\theta = \int_0^{hL} \beta(t)dt \leq \int_0^b \beta(t)dt = \Hyp(s).\]
    Hence, there exists $\theta \in [0,h)$ so that the summation $\sum_{j=1}^{L}\beta_j$ is at most $\Hyp(s)/h$. Thus we may fix such $\theta$. By \ref{thm:treerealization}, some realization of the sampled pivots yields a tree whose approximation error satisfies the stated bound. Consequently,
    \[ \Tree(s) \leq h+\lambda h(2L+1)\left(\sum_{j=1}^{L}\beta_j \right)
        +\frac{hL}{e\lambda}+\varepsilon \leq h + \lambda \Hyp(s) \left(\frac{2b}{h}+1 \right) + \frac{b}{e \lambda} + \varepsilon . \]
    Letting $\varepsilon \searrow 0$ completes the proof.
\end{proof}
We derive the main theorem by optimizing the parameters.
\begin{proof}[Proof of Theorem~\ref{thm:main} from Theorem~\ref{thm:treebound}]
    We first assume $b = 1$ by rescaling $s$. We will further assume that $\Hyp(s) \leq e/63$, since otherwise the trivial bound $\Tree(s) \leq 1$ already shows the desired estimate.
    
    Suppose $\Hyp(s) > 0$. Set
    \( h := \left( \frac{7 \Hyp(s)}{3e} \right)^{1/3} \; \text{and} \; \lambda:= \frac{1}{eh}.\)
    Since $h \leq \frac{1}{3}$, we obtain
    \begin{align*}
    \Tree(s)
    &\leq h + \lambda \Hyp(s) \left(1 + \frac{2}{h} \right) + \frac{1}{e\lambda}\\
    &\leq h+\frac{7\lambda}{3h}\Hyp(s)+\frac{1}{e\lambda}\\
    &=h+e\lambda h^2+\frac{1}{e\lambda}\\
    &=3h
    =\left(\frac{63}{e}\right)^{1/3}\Hyp(s)^{1/3}.
    \end{align*}
    If $\Hyp(s) = 0$, then Theorem~\ref{thm:main} gives $\Tree(s) \leq h + \frac{1}{e\lambda}$. Letting $h \searrow 0$ with $\lambda \to \infty$ proves $\Tree(s) = 0$, as desired.
\end{proof}

\begin{proof}[Proof of Corollary~\ref{thm:ultrametric-fitting} from Theorem~\ref{thm:main}]
Equip $X$ with the uniform distribution and set $s:=M-d$. Then
    \begin{align*}
        n^3 \Hyp(s) & = \sum_{(x,y,z) \in X^3} ( \min(s(x,z) , s(y,z)) - s(x,y))_+ \\
        & = \sum_{(x,y,z) \in X^3} (d(x,y) - \max(d(x,z) , d(y,z)))_{+} = \binom{n}{3} \operatorname{AvgUM}(d).
    \end{align*}
By Theorem~\ref{thm:main}, the ultrametric defined by $d_U(x,y)=M-s_T(x,y)$ for $x\ne y$ and $d_U(x,x)=0$ satisfies
    \[
    \EE\!\left[\sum_{\{x,y\}\in\binom{X}{2}}
    |d(x,y)-d_U(x,y)|\right]
    \le \frac{n^2}{2}\left(\frac{63}{e}M^2\Hyp(s)\right)^{1/3}
    = n\left(\frac{63}{4e}M^2\binom{n}{3}
    \operatorname{AvgUM}(d)\right)^{1/3}.
    \]
Note that the $\varepsilon$ terms disappear since its perturbations affect only the diagonal and hence do not contribute to the fitting error. Repeated pivots do not change the construction, so only their first occurrences need to be retained. This ordered list can be sampled directly by including each vertex independently with probability $1-e^{-\lambda/n}$ and ordering the selected vertices uniformly at random. By Poisson thinning and symmetry, this gives the same distribution as the original procedure after removing repetitions. Using this single list across all thresholds, there are at most $n$ pivots and $\binom{n}{2}+1$ relevant similarity levels. Thus clustering and refinement take polynomial time, regardless of the size of $\lambda$.
\end{proof}

\section{Discussion}\label{sec:discussion}

Our upper bound raises the question of the optimal dependence of
$\Tree(s)$ on $\Hyp(s)$ and the range bound $b$. It is natural to ask what the asymptotically tight bound is. A simple family shows that the square-root scale in terms of $b$ and $\Hyp(s)$ is unavoidable.
\begin{prop}
    Let $S$ be an atomless probability space with a measurable partition $U_1, V_1, \cdots, U_m, V_m$, each having mass $1/(2m)$. Define $s_m$ by
    \[
    s_m(x,y)=
    \begin{cases}
    b, & \text{if } (x,y) \in U_i \times V_i \text{ or } (x,y) \in V_i \times U_i
    \text{ for some } i,\\
    b \cdot \mathbf{1}(x=y), & \text{otherwise}.
    \end{cases}
    \]
    Then \( \Tree(s_m) = \frac{b}{2m} = \sqrt{b \Hyp(s_m)}.\)
\end{prop}
\begin{proof}
    It is easy to compute $\Hyp(s_m)$. Approximating $s_m = \varepsilon$ also yields $\Tree(s_m) \leq \frac{b}{2m} + \varepsilon$ for any $\varepsilon > 0$. To verify the lower bound $\Tree(s_m) \geq \frac{b}{2m}$, consider any class $D$
    of a threshold relation of a tree representation, and write
    $a_i=\PP(D\cap U_i)$ and $c_i=\PP(D\cap V_i)$. Relative to the zero
    relation, declaring all pairs in $D$ positive changes the
    disagreement probability by
    \[
        \PP(D)^2-4\sum_{i=1}^m a_i c_i
        =
        \sum_{i=1}^m(a_i-c_i)^2
        +2\sum_{i<j}(a_i+c_i)(a_j+c_j)
        \geq 0.
    \]
    Thus no threshold clustering improves the approximation error.
\end{proof}
In particular, the family shows that the assumption on $\sup s$ is essential and no universal estimate $\Tree(s)\leq Cb^{1-\alpha}\Hyp(s)^\alpha$ can hold with $\alpha>1/2$ as well. Together with our upper bound, this leaves a gap between the exponents $1/3$ and $1/2$. It is natural to ask whether the square-root estimate
\[ \Tree(s)\leq C\sqrt{b\,\Hyp(s)} \]
holds throughout the present framework.

One possible route to this estimate is to improve the passage from the relations $H_j$ to the hierarchy $K_j$. In our analysis, passing from $\{H_j\}$ to $\{K_j\}$ incurs a factor of order $L$ in the error bound. If this loss could be reduced to a constant factor, the procedure would yield $\Tree(s)=O(\sqrt{b\,\Hyp(s)})$. However, it turns out that the common refinement step can increase the total threshold disagreement by an arbitrarily large factor. The following Figure~\ref{fig:refinement-amplification} gives an amplification of $\varepsilon^{-1}-1$ with probability tending to one as $\varepsilon \to 0$. 

\begin{figure}[ht]
\centering
\begin{tikzpicture}[font=\small, line cap=round, line join=round,
  block/.style={circle,draw,fill=white,minimum size=6.5mm,inner sep=0pt},
  cluster/.style={draw,rounded corners=3pt,fill=black!3,line width=.55pt}]
  \node[font=\small\itshape] at (1.5,2.25) {Input similarities};
  \node[block] (U) at (0,1.05) {$U$};
  \node[block] (V) at (3,1.05) {$V$};
  \node[block,fill=black!12] (W) at (1.5,-.8) {$W$};
  \draw[line width=.9pt] (U) -- node[above=2pt] {$1$} (V);
  \draw (U) -- node[left=3pt] {$2 \varepsilon$} (W);
  \draw (V) -- node[right=3pt] {$\varepsilon$} (W);
  \node[above=2pt] at (U.north) {$2p(1-p)$};
  \node[above=2pt] at (V.north) {$p^2$};
  \node[below=2pt] at (W.south) {$(1-p)^2$};
  \node[font=\footnotesize] at (1.5,-2.1) {$s=1$ within each block};
  \node[font=\small\itshape] at (8.2,2.25) {Partitions on the event $E$};
  \node at (6.5,1.65) {$\varepsilon < t \leq 2 \varepsilon$};
  \node at (10,1.65) {$t > 2 \varepsilon$};
  \node at (4.55,.65) {$H_j$};
  \node at (4.55,-.75) {$K_j$};
  \draw[cluster] (5.35,.31) rectangle (6.8,.99);
  \draw[cluster] (7,.31) rectangle (7.7,.99);
  \node at (5.7,.65) {$U$};
  \node at (6.45,.65) {$W$};
  \node at (7.35,.65) {$V$};
  \draw[cluster] (8.85,.31) rectangle (10.3,.99);
  \draw[cluster] (10.5,.31) rectangle (11.2,.99);
  \node at (9.2,.65) {$U$};
  \node at (9.95,.65) {$V$};
  \node at (10.85,.65) {$W$};
  \draw[cluster] (5.35,-1.09) rectangle (6.8,-.41);
  \draw[cluster] (7,-1.09) rectangle (7.7,-.41);
  \node at (5.7,-.75) {$U$};
  \node at (6.45,-.75) {$W$};
  \node at (7.35,-.75) {$V$};
  \foreach \x/\lab in {9.2/U,10.025/V,10.85/W}{
    \draw[cluster] (\x-.35,-1.09) rectangle (\x+.35,-.41);
    \node at (\x,-.75) {$\lab$};
  }
  \node[font=\footnotesize] at (8.2,-2.1) {$K_j=\bigcap_{i\le j}H_i$};
\end{tikzpicture}
\caption{A temporary separation becomes permanent under common refinement. The labels outside the three input nodes indicate their masses. On $E$, the raw partitions reunite $U$ and $V$ at higher levels, while the common refinements keep them apart. Under our sampling scheme, the first pivot lies in $W$ and a later pivot lies in $U\cup V$ with probability tending to one as $p\searrow0$. Choosing $\varepsilon=cp^{3/2}$ with $c>0$ sufficiently large ensures that the prescribed mesh satisfies $h\le\varepsilon$. Hence, for every shift $\theta$, there is a threshold $\theta+ih\in(\varepsilon,2\varepsilon]$, at which the $UW$ pairs are positive and the $VW$ pairs are negative.}
\label{fig:refinement-amplification}
\end{figure}
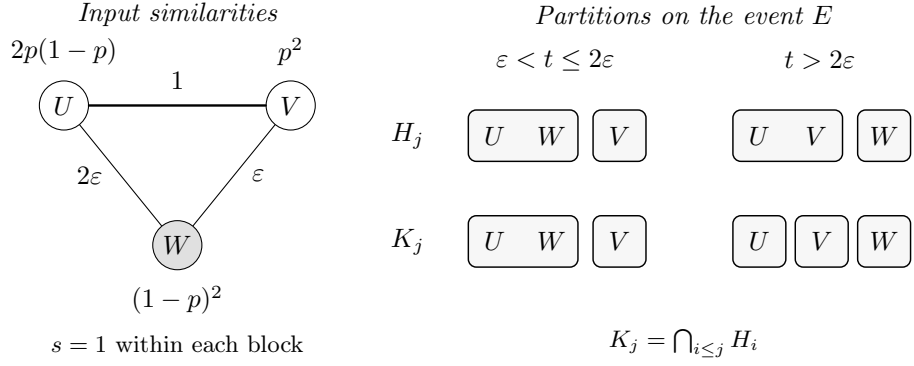

Considering a finite input case, such as $|S| = n$ with the uniform probability measure, is also interesting. For example, methods beyond pivoting may be relevant. LP rounding for hierarchical correlation clustering provides a concrete direction to explore~\cite{an2025handling, cohen2024fitting}. Such an approach would still require a bound relating the fractional objective to average hyperbolicity; an approximation guarantee relative to the optimal fitting cost alone does not establish the desired structural estimate.

Our bound can also be compared with the guarantees of \textsc{HCCUltraFit}~\cite{yim2024fitting}, which controls the total fitting error by the sum of the ultrametric triangle violations over unordered triples. Writing $\|d-d_U\|_1:=\sum_{\{x,y\}\in\binom{X}{2}}|d(x,y)-d_U(x,y)|$ for the total distortion, selecting the better of the two outputs gives
\[
    \EE\|d-d_U\|_1=O\left(\min\left\{n^3\operatorname{AvgUM}(d),\,n^2M^{2/3}\operatorname{AvgUM}(d)^{1/3}\right\}\right),
\]
where $M=\operatorname{diam}(X,d)$. At fixed $n$ and $M$, the linear estimate is stronger for sufficiently small $\operatorname{AvgUM}(d)$, whereas our cube-root estimate becomes stronger as this quantity increases.

Finite inputs offer additional choices of threshold levels. In our construction, the choice of mesh size $h$ makes the number of layers $L$ depend on $\Hyp(s)$. For $|S|=n$, however, there are only $O(n^2)$ distinct threshold relations, determined by the distinct values of $s$. Likewise, repeated occurrences of a sampled pivot can be discarded, leaving at most $n$ distinct pivots, although replacing i.i.d.\ sampling by sampling without replacement would require modifying our analysis. These observations alone do not improve our estimate: the lengths of the threshold intervals must be retained as weights in the error analysis, and changing the sampling scheme requires new probability estimates. It remains to determine whether adapting the threshold levels and pivot sampling to finite inputs can improve the bounds.

An algorithmic implementation also raises the question of parameter selection. Under the same finite setup, $\Hyp(s)$ can be computed exactly by summing over ordered triples in $\Theta(n^3)$ time. For large inputs, this is costly so that sampling triples instead may be preferable. The related question is then to select $h$ and $\lambda$ reliably using an estimate $\widetilde{\Hyp}(s)$: especially when $\Hyp(s)$ is small. Finding a shift $\theta\in[0,h)$ that satisfies the required bound is also nontrivial in practice.

Finally, our result also relates to the spin-glass setting that motivated Chatterjee and Sloman~\cite{chatterjee2021average}. Their application assumes asymptotic ultrametricity of overlaps, or of a bounded transformation of them, and yields hierarchical approximations in mean. Related results include Panchenko's ultrametricity theorem~\cite{panchenko2013parisi} and Jagannath's approximate hierarchical decompositions~\cite{jagannath2017approximate}. The additional structure of these models may allow improved bounds in terms of average hyperbolicity.

\subsection*{Acknowledgements and AI Disclosure}
We thank Anna Gilbert, Hyung-Chan An and Changyeol Lee for their valuable comments and feedback. We also acknowledge the use of OpenAI's GPT-5.6 for assistance with language editing, figure generation, and exploratory discussion.

\bibliography{ref}
\bibliographystyle{plain}

\end{document}